\documentclass[12pt]{article}

\title{\vspace{-8pt}Homological vanishing for the Steinberg representation}
\author{Avner Ash \and Andrew Putman\thanks{AP is supported in part by NSF grants DMS-1255350 and DMS-1737434.} \and Steven V Sam\thanks{SS is supported in part by NSF grant DMS-1500069.}}
\date{November 20, 2017}

\usepackage{tocloft}
\usepackage{bbm}
\newcommand{\arxiv}[1]{\href{http://arxiv.org/abs/#1}{{\tt arXiv:#1}}}

\usepackage{amsmath,amssymb,amsthm,amscd,amsfonts}
\usepackage{mathtools}
\usepackage{epsfig,pinlabel,paralist}
\usepackage[vmargin=1in, hmargin=1in]{geometry}
\usepackage[font=small,format=plain,labelfont=bf,up,textfont=it,up]{caption}
\usepackage{type1cm}
\usepackage{calc}
\usepackage{enumerate}
\usepackage{bm}
\usepackage{microtype}
\usepackage[all,cmtip]{xy}
\usepackage{eucal}
\usepackage{paralist}
\usepackage{tikz}
\usetikzlibrary{calc}
\usepackage{lmodern}
\usepackage[T1]{fontenc}

\usepackage[bookmarks, bookmarksdepth=2, colorlinks=true, linkcolor=blue, citecolor=blue, urlcolor=blue]{hyperref}

\usepackage{etoolbox}
\apptocmd{\thebibliography}{\raggedright}{}{}

\numberwithin{equation}{section}

\theoremstyle{plain}
\newtheorem{theorem}{Theorem}[section]

\newtheorem{lemma}[theorem]{Lemma}

\newtheorem{corollary}[theorem]{Corollary}
\newtheorem{conjecture}[theorem]{Conjecture}

\theoremstyle{definition}

\newtheorem{defn}[theorem]{Definition}
\newenvironment{definition}[1][]{\begin{defn}[#1]\pushQED{\qed}}{\popQED \end{defn}}

\newtheorem{rmk}[theorem]{Remark}
\newenvironment{remark}[1][]{\begin{rmk}[#1] \pushQED{\qed}}{\popQED \end{rmk}}

\newtheorem{eg}[theorem]{Example}

\DeclareMathOperator{\Image}{Im}

\DeclareMathOperator{\Sp}{Sp}
\DeclareMathOperator{\SO}{SO}
\DeclareMathOperator{\GL}{GL}
\DeclareMathOperator{\SL}{SL}
\DeclareMathOperator{\PSL}{PSL}

\newcommand\Proj{\ensuremath{\mathbb{P}}}

\newcommand\C{\ensuremath{\mathbb{C}}}
\newcommand\Z{\ensuremath{\mathbb{Z}}}
\newcommand\Q{\ensuremath{\mathbb{Q}}}

\newcommand\Field{\ensuremath{\mathrm{k}}}

\DeclareMathOperator{\HH}{H}

\newcommand\RH{\ensuremath{\widetilde{\HH}}}

\newcommand\Span[1]{\ensuremath{\langle #1 \rangle}}

\newcommand\Set[2]{\ensuremath{\{\text{#1 $|$ #2}\}}}

\newcommand\GG{\ensuremath{\mathbf{G}}}
\newcommand\Para{\ensuremath{\mathbf{PG}}}
\newcommand\Levi{\ensuremath{\mathbf{LG}}}
\newcommand\Uni{\ensuremath{\mathbf{UG}}}
\newcommand\ParaGL{\ensuremath{\mathbf{PGL}}}

\newcommand\UniGL{\ensuremath{\mathbf{UGL}}}
\newcommand\Stab{\ensuremath{\mathbf{FG}}}
\newcommand\Cpx{\ensuremath{\mathbf{CG}}}
\newcommand\hpartial{\ensuremath{\widehat{\partial}}}
\newcommand\hkappa{\ensuremath{\widehat{\kappa}}}

\newcommand\Tits{\ensuremath{\mathcal{T}}}
\DeclareMathOperator{\St}{St}

\DeclareMathOperator{\Res}{Res}
\DeclareMathOperator{\Ind}{Ind}

\newcommand\Apartment[1]{\ensuremath{\begin{Vmatrix} #1 \end{Vmatrix}}}

\begin{document}

\maketitle

\vspace{-24pt}
\begin{abstract}
\noindent
For a field $\Field$, we prove that the $i$th homology of the groups $\GL_n(\Field)$, $\SL_n(\Field)$, $\Sp_{2n}(\Field)$, $\SO_{n,n}(\Field)$, and $\SO_{n,n+1}(\Field)$ with coefficients in their Steinberg representations vanish for $n \geq 2i+2$.
\end{abstract}

\section{Introduction}

Let $\GG$ be a connected reductive group over a field $\Field$.  A basic geometric object associated to $\GG(\Field)$
is its Tits building.  By definition, this is the simplicial complex $\Tits_{\GG}(\Field)$ whose $i$-simplices
are increasing sequences
\[
0 \subsetneq P_0 \subsetneq \cdots \subsetneq P_i \subsetneq \GG(\Field)
\]
of parabolic $\Field$-subgroups of $\GG(\Field)$.  Letting $r$ be the semisimple $\Field$-rank of $\GG$, the complex
$\Tits_{\GG}(\Field)$ is $(r-1)$-dimensional, and the Solomon--Tits theorem \cite[Theorem IV.5.2]{BrownBuildings}
says that in fact $\Tits_{\GG}(\Field)$ is homotopy equivalent to a wedge of $(r-1)$-dimensional spheres.  Letting
$R$ be a commutative ring, the {\em Steinberg representation} of $\GG(\Field)$ over $R$, denoted $\St_{\GG}(\Field;R)$,
is $\RH_{r-1}(\Tits_{\GG}(\Field);R)$.
This is one of the most important representations of $\GG(\Field)$; for instance, if
$\GG$ is any of the classical groups in Theorem \ref{theorem:vanishing} below (e.g.\ $\GG = \SL_n$) and 
 $\Field$ is a finite field
of characteristic $p$, then $\St_{\GG}(\Field;\C)$ is the unique irreducible representation of $\GG(\Field)$ whose dimension is a positive power of $p$ (see \cite{MZ}, which proves this aside from three small cases that must be checked by hand). 
  See \cite{HumphreysSurvey} for a survey of many results concerning the Steinberg representation.
  
The twisted homology groups $\HH_i(\GG(\Field);\St_{\GG}(\Field;R))$ play an interesting role
in algebraic K-theory; see \cite[Theorem 3]{QuillenFinite}.
If $\GG(\Field)$ is a finite group of Lie type, then $\St_{\GG}(\Field;\Field)$ is a
projective $\GG(\Field)$-module (see \cite{HumphreysSurvey}), and thus
the homology groups
$\HH_i(\GG(\Field);\St_{\GG}(\Field;\Field))$ all vanish.  However, it is definitely
not the case that $\St_{\GG}(\Field;R)$ is projective for a general commutative ring $R$,
and if $\Field$ is an infinite field then  $\St_{\GG}(\Field;\Field)$ need not  be projective.  Our main theorem says that nevertheless for the classical groups, the
homology groups $\HH_i(\GG(\Field);\St_{\GG}(\Field;R))$ always vanish in a stable range.

\begin{theorem} \label{theorem:vanishing}
Let $\GG_n$ be either $\GL_{n}$, $\SL_{n}$, $\Sp_{2n}$, $\SO_{n,n}$, or $\SO_{n,n+1}$.
Then for all fields $\Field$ and all commutative rings $R$, we have
$\HH_i(\GG_n(\Field);\St_{\GG_n}(\Field;R)) = 0$ for $n \geq 2i+2$. Furthermore, there exists a surjection
$\HH_i(\GG_{2i}(\Field);\St_{\GG_{2i}}(\Field;R)) \to \HH_i(\GG_{2i+1}(\Field);\St_{\GG_{2i+1}}(\Field;R))$.
\end{theorem}

\begin{remark}
When $\text{char}(\Field)=2$, the groups $\SO_{n,n}(\Field)$ and $\SO_{n,n+1}(\Field)$ in
Theorem \ref{theorem:vanishing} are to be taken naively as the stabilizers of appropriate quadratic
forms (see \S \ref{section:varioussubgroups} below); we ignore the Dickson invariant.
\end{remark}

Theorem \ref{theorem:vanishing} (and its proof)
is motivated by the following conjecture of Church--Farb--Putman.
Recall that Borel--Serre \cite{BorelSerreCorners} proved
that the virtual cohomological dimension of $\SL_n(\Z)$ is $\binom{n}{2}$.

\begin{conjecture}[{\cite[Conjecture 2]{ChurchFarbPutmanConjecture}}]
\label{conjecture:vanishing1}
We have $\HH^{\binom{n}{2}-i}(\SL_n(\Z);\Q) = 0$ for $n \geq i+2$.
\end{conjecture}

\noindent
In other words, the rational cohomology of $\SL_n(\Z)$ vanishes in codimension $i$ as long as $n$ is
sufficiently large. Conjecture \ref{conjecture:vanishing1} was proved for $i=0$ by Lee--Szczarba \cite{LeeSzczarba} and
for $i=1$ by Church--Putman \cite{ChurchPutmanCodimensionOne}.  It is open for $i \geq 2$.

To connect Conjecture \ref{conjecture:vanishing1} to Theorem \ref{theorem:vanishing}, recall
that Borel--Serre \cite{BorelSerreCorners}
proved that $\SL_n(\Z)$ satisfies a version of Poincar\'{e}--Lefschetz duality called
virtual Bieri--Eckmann duality.   This duality involves a ``dualizing module'' that measures
the ``homology at infinity''.  In our situation, that dualizing module is the Steinberg representation
$\St_{\SL_n}(\Q;\Q)$ and we have
\[
\HH^{\binom{n}{2}-i}(\SL_n(\Z);\Q) = \HH_i(\SL_n(\Z);\St_{\SL_n}(\Q;\Q)).
\]
Conjecture~\ref{conjecture:vanishing1} is thus equivalent to the following conjecture, which resembles Theorem~\ref{theorem:vanishing} for $\GG_n = \SL_n$.

\begin{conjecture}
\label{conjecture:vanishing2}
We have $\HH_i(\SL_n(\Z);\St_{\SL_n}(\Q;\Q)) = 0$ for $n \geq i+2$.
\end{conjecture}

\begin{remark}
The proofs by Lee--Szczarba \cite{LeeSzczarba} and Church--Putman \cite{ChurchPutmanCodimensionOne}
of special cases of Conjecture \ref{conjecture:vanishing1} both start by translating things into the language
of Conjecture \ref{conjecture:vanishing2}.
\end{remark}

We now briefly describe our proof of Theorem \ref{theorem:vanishing}.
As we will discuss in \S \ref{section:stability} below, there is a natural inclusion
$\St_{\GG_{n-1}}(\Field;R) \rightarrow \St_{\GG_n}(\Field;R)$.  
This induces a stabilization map
\begin{equation}
\label{eqn:stabilizationintro}
\HH_i(\GG_{n-1}(\Field);\St_{\GG_{n-1}}(\Field;R)) \rightarrow 
\HH_i(\GG_{n}(\Field);\St_{\GG_{n}}(\Field;R)).
\end{equation}
We will show in \S \ref{section:stability} that to prove that 
$\HH_i(\GG_{n}(\Field);\St_{\GG_{n}}(\Field;R)) = 0$ for
large $n$, it is enough to prove the seemingly weaker assertion that \eqref{eqn:stabilizationintro} is
a surjection for large $n$.
This idea was first introduced by Church--Farb--Putman \cite{ChurchFarbPutmanConjecture}
as a strategy for proving Conjecture \ref{conjecture:vanishing2}.  It was also noticed by Ash in unpublished work. 

The surjectivity of \eqref{eqn:stabilizationintro} is a weak form of {\em homological stability}.
There is an enormous literature on homological stability theorems.  The basic technique underlying most results 
in the subject goes back to unpublished work of Quillen.  In \cite{DwyerStability}, Dwyer used these ideas to prove a twisted
homological stability theorem for $\GL_n(\Field)$ with quite general coefficient systems.  This work 
was later generalized by van der Kallen \cite{VDK} and very recently by Randal-Williams--Wahl
\cite{RandalWilliamsWahl}, whose results cover all the classical groups in
Theorem \ref{theorem:vanishing}.  Unfortunately, the Steinberg representation does {\em not} satisfy the conditions
in any of these known theorems.  Indeed, these theorems are general enough that if it did, then this
would quickly lead to a proof of Conjecture \ref{conjecture:vanishing2}.  Nevertheless, we are able to use
some delicate properties of the Steinberg representation to jury-rig the Quillen machine such that it works
to prove that \eqref{eqn:stabilizationintro} is surjective for large $n$.

\begin{remark}
Homological stability for a sequence of groups and homomorphisms $X_1 \to X_2 \to \cdots$ states that the induced maps $\HH_i(X_n) \to \HH_i(X_{n+1})$ are isomorphisms for $n \gg 0$. Alternatively, we can think of each map as ``multiplication by $t$'' and give $\bigoplus_n \HH_i(X_n)$ the structure of a $R[t]$-module, where $R$ denotes our coefficient ring. At least when $R$ is a field, this isomorphism would be a consequence of finite generation. 

In our setting, with homology twisted by the Steinberg representation, one should instead think of this map as ``multiplication by $t$'' where $t$ is a generator for the exterior algebra in one variable $R[t]/t^2$, so that the groups $\HH_i$ being $0$ for $n \gg 0$ would again be a consequence of finite generation. At least when $\Field$ is a finite field of size $q$ and $R$ is the field of complex numbers, this is consistent with the idea that $\GL_n({\bf F}_q)$ is a $q$-analogue of the symmetric group and the Steinberg representation is the $q$-analogue of its sign representation, which is made more precise via their connection to symmetric functions, see \cite[\S\S I.7, IV.4]{macdonald}.
\end{remark}

\paragraph{Outline.}
We begin in \S \ref{section:backgroundnotation} with some background and notation.
Next, in \S \ref{section:stability} we
reduce Theorem \ref{theorem:vanishing} to an appropriate homological stability theorem. 
We then prove a key isomorphism in \S \ref{section:stabilizers}.  We
prove Theorem \ref{theorem:vanishing} in \S \ref{section:vanishing}.  This proof depends on
a calculation which we perform in \S \ref{section:differential}.

\paragraph{Convention regarding the empty set.}
If $X$ is the empty set and $R$ is a commutative ring, then we define $\RH_{-1}(X;R) = R$.
With this convention, if the semisimple $\Field$-rank of $\GG$ is $0$, then
$\St_{\GG}(\Field;R) = R$ with the trivial $\GG(\Field)$-action. 

\paragraph{Acknowledgments.}
The second author would like to thank Thomas Church and Benson Farb for many inspiring conversations concerning
Conjectures \ref{conjecture:vanishing1} and \ref{conjecture:vanishing2}.

\section{Background and notation}
\label{section:backgroundnotation}

This section contains some background information and notation needed in the remainder of the paper.  It
consists of two subsections: \S \ref{section:varioussubgroups} introduces some distinguished parabolic
subgroups, and \S \ref{section:steinbergfacts} gives some background about the Steinberg representations.

Throughout this section, $\Field$ is a field and $\GG_n$ is 
either $\GL_{n}$, $\SL_{n}$, $\Sp_{2n}$, $\SO_{n,n}$, or $\SO_{n,n+1}$.

\subsection{Parabolic and stabilizer subgroups}
\label{section:varioussubgroups}

Our proof of Theorem~\ref{theorem:vanishing} depends on a careful study of various subgroups of $\GG_n(\Field)$.  In this section,
we will introduce notation for these subgroups: a certain parabolic subgroup $\Para_n^{\ell}(\Field)$, its unipotent radical $\Uni_n^{\ell}(\Field)$, a Levi component $\Levi_n^{\ell}(\Field)$ of $\Para_n^{\ell}(\Field)$, and another subgroup $\Stab_n^{\ell}(\Field)$ that lies in $\Para_n^{\ell}(\Field)$ and fixes certain vectors.

\paragraph{General and special linear groups.}
Assume first that $\GG_n$ is either $\GL_{n}$ or $\SL_{n}$.  The group $\GG_n(\Field)$ thus acts on
the vector space $\Field^{n}$, and the $\Field$-parabolic subgroups of $\GG_n(\Field)$ are the stabilizers of flags
of subspaces of $\Field^{n}$.  Let $(\vec{a}_1,\ldots,\vec{a}_{n})$ be the standard basis for $\Field^{n}$.
For $1 \leq \ell \leq n$, the group $\Para_n^{\ell}(\Field)$ is defined to be the $\GG_n(\Field)$-stabilizer of the flag
\[
0 \subsetneq \Span{\vec{a}_1,\ldots,\vec{a}_{\ell}}.
\]
The group $\Uni_n^{\ell}(\Field)$ is the subgroup of $\Para_n^{\ell}(\Field)$ consisting of all $M \in \Para_n^{\ell}(\Field)$ that act as the identity on both
\[\Span{\vec{a}_1,\ldots,\vec{a}_{\ell}} \quad \quad \text{and} \quad \quad
\Field^n / \Span{\vec{a}_1,\ldots,\vec{a}_{\ell}}.\]
The group $\Levi_n^{\ell}(\Field)$ is defined to be the $\Para_n^{\ell}(\Field)$-stabilizer of the flag
\[
0 \subsetneq \Span{\vec{a}_{\ell+1},\ldots,\vec{a}_n}.
\]
If $\GG_n = \GL_{n}$ then $\Levi_n^\ell(\Field)$ is the subgroup $\GL_{\ell}(\Field) \times \GL_{n-\ell}(\Field)$
of $\GG_n$, while if $\GG_n = \SL_{n}$ then $\Levi_n^{\ell}(\Field)$ is the subgroup of $\GL_{\ell}(\Field) \times \GL_{n-\ell}(\Field)$
consisting of matrices of determinant $1$.  Finally, define 
\[\Stab_n^{\ell}(\Field) = \Set{$M \in \GG_n(\Field)$}{$M(\vec{a}_j) = \vec{a}_j$ for $1 \leq j \leq \ell$}.\]
We thus have $\Stab_n^{\ell}(\Field) \subset \Para_n^{\ell}(\Field)$.

\paragraph{Symplectic groups.}
Now assume that $\GG_n = \Sp_{2n}$.  Letting $\omega(\cdot,\cdot)$ be the standard symplectic form on
$\Field^{2n}$, the group $\GG_n(\Field)$ is the subgroup of $\GL_n(\Field)$ consisting of elements
that preserve $\omega(\cdot,\cdot)$.  The $\Field$-parabolic subgroups of $\GG_n(\Field)$ are the
$\GG_n(\Field)$-stabilizers of flags of isotropic subspaces of $\Field^{2n}$, that is, subspaces on which
$\omega(\cdot,\cdot)$ vanishes identically.  Let $(\vec{a}_1,\ldots,\vec{a}_n,\vec{b}_1,\ldots,\vec{b}_n)$
be the standard symplectic basis for $\Field^{2n}$, so
\[
\omega(\vec{a}_j,\vec{a}_{j'}) = \omega(\vec{b}_j,\vec{b}_{j'}) = 0 \quad \text{and} \quad
\omega(\vec{a}_j,\vec{b}_{j'}) = \delta_{jj'}
\]
for $1 \leq j,j' \leq n$, where $\delta_{jj'}$ is the Kronecker delta function.  For $1 \leq \ell \leq n$, the
group $\Para_n^{\ell}(\Field)$ is defined to be the $\GG_n(\Field)$-stabilizer of the isotropic flag
\[
0 \subsetneq \Span{\vec{a}_1,\ldots,\vec{a}_{\ell}}.
\]
The group $\Uni_n^{\ell}(\Field)$ is the subgroup of $\Para_n^{\ell}(\Field)$ consisting of all $M \in \Para_n^{\ell}(\Field)$ that act as the identity on both
\begin{align*}
\Span{\vec{a}_1,\ldots,\vec{a}_{\ell}} \quad \quad \text{and} \quad \quad
&\Span{\vec{a}_1,\ldots,\vec{a}_{\ell}}^{\perp} / \Span{\vec{a}_1,\ldots,\vec{a}_{\ell}}\\
&\ \ \ \ \ = \Span{\vec{a}_1,\ldots,\vec{a}_{\ell},\vec{a}_{\ell+1},\vec{b}_{\ell+1},\ldots,\vec{a}_n,\vec{b}_n} / \Span{\vec{a}_1,\ldots,\vec{a}_{\ell}}.
\end{align*}
The group $\Levi_n^{\ell}(\Field)$ is defined to be the $\Para_n^{\ell}(\Field)$-stabilizer of the isotropic flag
\[
0 \subsetneq \Span{\vec{b}_1,\ldots,\vec{b}_{\ell}}.
\]
The group $\Levi_n^{\ell}(\Field)$ is thus isomorphic to $\GL_{\ell}(\Field) \times \GG_{n-\ell}(\Field)$.  Finally,
define
\[
\Stab_n^{\ell}(\Field) = \Set{$M \in \GG_n(\Field)$}{$M(\vec{a}_j) = \vec{a}_j$ for $1 \leq j \leq \ell$}.
\]
We thus have $\Stab_n^{\ell}(\Field) \subset \Para_n^{\ell}(\Field)$.

\paragraph{Orthogonal groups.}
Finally, assume that $\GG_n$ is either $\SO_{n,n}$ or $\SO_{n,n+1}$.  For an appropriate $m$, the
group $\GG_n(\Field)$ is then the subgroup of $\SL_m(\Field)$ consisting of elements
that preserve a quadratic form $q(\cdot)$ on $\Field^m$:
\begin{compactitem}
\item If $\GG_n = \SO_{n,n}$, then let $m = 2n$ and let $(\vec{a}_1,\ldots,\vec{a}_n,\vec{b}_1,\ldots,\vec{b}_n)$
be the standard basis for $\Field^m$.  The group $\GG_n(\Field)$ is the $\SL_m(\Field)$-stabilizer of the quadratic form
$q(\cdot)$ on $\Field^m$ defined via the formula
\[
q\left(\sum_{j=1}^n \left(c_j \vec{a}_j + d_j \vec{b}_j\right)\right) = \sum_{j=1}^n c_j d_j \quad \quad (c_j,d_j \in \Field).
\]
\item If $\GG_n = \SO_{n,n+1}$, then let $m=2n+1$ and let $(\vec{a}_1,\ldots,\vec{a}_n,\vec{b}_1,\ldots,\vec{b}_{n},\vec{e})$ 
be the standard basis for $\Field^m$.  The group $\GG_n(\Field)$ is the $\SL_m(\Field)$-stabilizer of the quadratic form
$q(\cdot)$ on $\Field^m$ defined via the formula
\[
q\left(\lambda \vec{e} + \sum_{j=1}^n \left(c_j \vec{a}_j + d_j \vec{b}_j\right)\right) = \lambda^2 + \sum_{j=1}^n c_j d_j \quad \quad (c_j, d_j, \lambda \in \Field).
\]
\end{compactitem}
In both cases, the $\Field$-parabolic subgroups of $\GG_n(\Field)$
are the $\GG_n(\Field)$-stabilizers of flags of isotropic subspaces of $\Field^m$, that is, subspaces on which
$q(\cdot)$ vanishes identically.  For $1 \leq \ell \leq n$ the group $\Para_n^{\ell}(\Field)$ is defined
to be the $\GG_n(\Field)$-stabilizer of the isotropic flag
\[
0 \subsetneq \Span{\vec{a}_1,\ldots,\vec{a}_{\ell}}.
\]
For $\GG_n = \SO_{n,n}$, 
the group $\Uni_n^{\ell}(\Field)$ is the subgroup of $\Para_n^{\ell}(\Field)$ consisting of all $M \in \Para_n^{\ell}(\Field)$ that act as the identity on both
\begin{align*}
\Span{\vec{a}_1,\ldots,\vec{a}_{\ell}} \quad \quad \text{and} \quad \quad
&\Span{\vec{a}_1,\ldots,\vec{a}_{\ell}}^{\perp} / \Span{\vec{a}_1,\ldots,\vec{a}_{\ell}}\\
&\ \ \ \ \ = \Span{\vec{a}_1,\ldots,\vec{a}_{\ell},\vec{a}_{\ell+1},\vec{b}_{\ell+1},\ldots,\vec{a}_n,\vec{b}_n} / \Span{\vec{a}_1,\ldots,\vec{a}_{\ell}},
\end{align*}
while if $\GG_n = \SO_{n,n+1}$, then the group $\Uni_n^{\ell}(\Field)$ is the subgroup of $\Para_n^{\ell}(\Field)$ consisting
of all $M \in \Para_n^{\ell}(\Field)$ that act as the identity on both
\begin{align*}
\Span{\vec{a}_1,\ldots,\vec{a}_{\ell}} \quad \quad \text{and} \quad \quad
&\Span{\vec{a}_1,\ldots,\vec{a}_{\ell}}^{\perp} / \Span{\vec{a}_1,\ldots,\vec{a}_{\ell}}\\
&\ \ \ \ \ = \Span{\vec{a}_1,\ldots,\vec{a}_{\ell},\vec{a}_{\ell+1},\vec{b}_{\ell+1},\ldots,\vec{a}_n,\vec{b}_n,\vec{e}} / \Span{\vec{a}_1,\ldots,\vec{a}_{\ell}}.
\end{align*}
The group $\Levi_n^{\ell}(\Field)$ is defined to be the $\Para_n^{\ell}(\Field)$-stabilizer of the isotropic flag
\[
0 \subsetneq \Span{\vec{b}_1,\ldots,\vec{b}_{\ell}}.
\]
The group $\Levi_n^{\ell}(\Field)$ is thus isomorphic to $\GL_{\ell}(\Field) \times \GG_{n-\ell}(\Field)$.  Finally,
define
\[
\Stab_n^{\ell}(\Field) = \Set{$M \in \GG_n(\Field)$}{$M(\vec{a}_j) = \vec{a}_j$ for $1 \leq j \leq \ell$}.
\]
We thus have $\Stab_n^{\ell}(\Field) \subset \Para_n^{\ell}(\Field)$.

\subsection{Facts about the Steinberg representation}
\label{section:steinbergfacts}

Let $R$ be a commutative ring. 
The following theorem of Reeder \cite{ReederSteinberg} will play an important role in our proof of
Theorem \ref{theorem:vanishing}.

\begin{theorem}[{\cite[Proposition 1.1]{ReederSteinberg}}]
\label{theorem:reeder}
Let $\GG$ be a connected reductive group defined over a field $\Field$, let $\Para(\Field)$ be a
$\Field$-parabolic subgroup of $\GG(\Field)$, and let $\Levi(\Field)$ be a Levi component of
$\Para(\Field)$.  Then there exists an $\Levi(\Field)$-equivariant map
\[
\St_{\Levi}(\Field;R) \longrightarrow \Res^{\GG(\Field)}_{\Levi(\Field)} \St_{\GG}(\Field;R)
\]
such that the induced map
\[
\Ind_{\Levi(\Field)}^{\Para(\Field)} \St_{\Levi}(\Field;R) \rightarrow \Res^{\GG(\Field)}_{\Para(\Field)} \St_{\GG}(\Field;R)
\]
is an isomorphism.
\end{theorem}

\begin{remark}
The map in Theorem \ref{theorem:reeder} is not unique; for instance, it can be
post-composed with any element of the unipotent radical of $\Para(\Field)$.  The
paper \cite{ReederSteinberg} contains a specific construction of this map, and whenever
we refer to the map in Theorem \ref{theorem:reeder} we mean the one constructed
in \cite{ReederSteinberg}.  
\end{remark}

We wish to apply this to the distinguished parabolic subgroups $\Para_n^{\ell}(\Field)$ that
we introduced in \S \ref{section:varioussubgroups}.  To do this, we need to identify
$\St_{\Levi_n^{\ell}}(\Field;R)$.

\begin{lemma}
\label{lemma:identifystlevi}
Let $\GG_n$ be either $\GL_{n}$, $\SL_{n}$, $\Sp_{2n}$, $\SO_{n,n}$, or $\SO_{n,n+1}$.
Then for all fields $\Field$ and all commutative rings $R$, we have
\[
\St_{\Levi_n^{\ell}}(\Field;R) = \St_{\GL_{\ell}}(\Field;R) \otimes \St_{\GG_{n-\ell}}(\Field;R)
\]
for $1 \leq \ell \leq n$.
\end{lemma}

\begin{proof}
For $\GG_n \neq \SL_n$, this follows from the decomposition 
$\Levi_n^{\ell}(\Field) = \GL_{\ell}(\Field) \times \GG_{n-\ell}(\Field)$.  For $\GG_n = \SL_n$, we
instead have that $\Levi_n^{\ell}(\Field)$ is the subgroup of 
$\GL_{\ell}(\Field) \times \GL_{n-\ell}(\Field)$ consisting of matrices of determinant $1$.  The
lemma in this case follows from two facts:
\begin{compactitem}
\item There is a bijection between $\Field$-parabolic subgroups
of $\GL_{n-\ell}(\Field)$ and $\SL_{n-\ell}(\Field)$, and thus an $\SL_{n-\ell}(\Field)$-equivariant isomorphism between $\St_{\GL_{n-\ell}}(\Field; R)$ and $\St_{\SL_{n-\ell}}(\Field; R)$.  
\item There is a bijection between $\Field$-parabolic subgroups of 
$\GL_{\ell}(\Field) \times \GL_{n-\ell}(\Field)$ and $\Levi_n^{\ell}(\Field)$, and thus an
$\Levi_n^{\ell}(\Field)$-equivariant isomorphism between
$\St_{\GL_{\ell}}(\Field;R) \otimes \St_{\GL_{n-\ell}}(\Field;R)$ and $\St_{\Levi_n^{\ell}}(\Field;R)$.
\end{compactitem}
Both of these bijections come from taking intersections.
\end{proof}

These two results allow us to make the following definition.

\begin{definition}
Let $\GG_n$ be either $\GL_{n}$, $\SL_{n}$, $\Sp_{2n}$, $\SO_{n,n}$, or $\SO_{n,n+1}$.
Also, let $\Field$ be a field and $R$ be a commutative ring.  For $1 \leq \ell \leq n$,
the {\em Reeder product map} is the map 
\[\St_{\GL_{\ell}}(\Field;R) \otimes \St_{\GG_{n-\ell}}(\Field;R) \longrightarrow \St_{\GG_n}(\Field;R)\]
obtained by combining Lemma \ref{lemma:identifystlevi} and Theorem \ref{theorem:reeder}.
\end{definition}

\begin{remark}
\label{remark:reederdecompose}
Identifying $\St_{\GL_{\ell}}(\Field;R) \otimes \St_{\GG_{n-\ell}}(\Field;R)$ with its
image in $\St_{\GG_n}(\Field;R)$ under the Reeder product map, one way of viewing
Theorem \ref{theorem:reeder} is that it asserts that
\[\St_{\GG_n}(\Field;R) = \bigoplus_{u \in \Uni_n^{\ell}(\Field)} u \cdot (\St_{\GL_{\ell}}(\Field;R) \otimes \St_{\GG_{n-\ell}}(\Field;R)). \qedhere\]
\end{remark}

We will also need the following lemma, which is precisely the case $i=0$ of
Theorem \ref{theorem:vanishing}.  It generalizes \cite[Theorem 4.1]{LeeSzczarba}.  Recall
that if $G$ is a group and $M$ is a $G$-module, then the coinvariants $M_G$ are the largest quotient of $M$ on
which $G$ acts trivially.  The coinvariants $M_G$ are isomorphic to $\HH_0(G;M)$.

\begin{lemma}
\label{lemma:killcoinvariants}
Let $\GG_n$ be either $\GL_{n}$, $\SL_{n}$, $\Sp_{2n}$, $\SO_{n,n}$, or $\SO_{n,n+1}$.
Then for all fields $\Field$ and all commutative rings $R$, we have
$(\St_{\GG_n}(\Field;R))_{\GG_n(\Field;R)} = 0$ for $n \geq 2$.
\end{lemma}

\begin{proof}
Theorem~\ref{theorem:reeder} implies that
\[
\Ind_{\Levi_n^2(\Field)}^{\Para_n^2(\Field)} \St_{\GL_2}(\Field;R) \otimes \St_{\GG_{n-2}}(\Field;R) \cong \Res^{\GG_n(\Field)}_{\Para_n^2(\Field)} \St_{\GG_n}(\Field;R).
\]
It is thus enough to prove that
\[
(\St_{\GL_2}(\Field;R) \otimes \St_{\GG_{n-2}}(\Field;R))_{\Levi_n^2(\Field)} = 0.
\]
Whatever $\GG_n$ is, the group $\Levi_n^2(\Field)$ contains the subgroup $\SL_2(\Field) \times 1$.  It is thus
enough to prove that
\[
(\St_{\GL_2}(\Field;R))_{\SL_2(\Field)} = 0.
\]
This is an easy exercise using the fact that 
\[\St_{\GL_2}(\Field;R) = \RH_0(\Tits_2(\Field);R) = \RH_{0}(\Proj^1(\Field);\Q),\]
where $\Proj^1(\Field)$ is the projective line over $\Field$, regarded as a discrete set of points. For details, see \cite[Theorem 4.1]{LeeSzczarba}.
\end{proof}

\section{Reduction to stability}
\label{section:stability}

Let $\GG_n$ be either $\GL_{n}$, $\SL_{n}$, $\Sp_{2n}$, $\SO_{n,n}$, or $\SO_{n,n+1}$. 
Let $\Field$ be a field and $R$ be a commutative
ring.  In this section, we reduce Theorem \ref{theorem:vanishing} to an appropriate
homological stability theorem.

Fix some $i \geq 0$ and some $n \geq 2$.  The {\em stabilization map} for 
$\HH_i(\GG_{n-1}(\Field);\St_{\GG_{n-1}}(\Field;R))$ is the map
\begin{equation}
\label{eqn:stabmap}
\HH_i(\GG_{n-1}(\Field);\St_{\GG_{n-1}}(\Field;R)) \rightarrow \HH_i(\GG_n(\Field);\St_{\GG_n}(\Field;R))
\end{equation}
induced by the following two maps:
\begin{compactitem}
\item The group homomorphism $\GG_{n-1}(\Field) \rightarrow \GG_n(\Field)$ obtained
as follows.  The group $\Levi_n^1(\Field)$ is a subgroup of 
$\GL_1(\Field) \times \GG_{n-1}(\Field)$ that contains the subgroup
$1 \times \GG_{n-1}(\Field)$.  In fact, 
$\Levi_n^1(\Field) = \GL_1(\Field) \times \GG_{n-1}(\Field)$ except
when $\GG_n = \SL_n$.  We can thus define a homomorphism $\GG_{n-1}(\Field) \rightarrow \GG_n(\Field)$
via the composition
\[\GG_{n-1}(\Field) = 1 \times \GG_{n-1}(\Field) \hookrightarrow \Levi_n^1(\Field) \hookrightarrow \GG_n(\Field).\]
\item The map $\St_{\GG_{n-1}}(\Field;R) \rightarrow \St_{\GG_{n}}(\Field;R)$ that
equals the composition
\[\St_{\GG_{n-1}}(\Field;R) \cong R \otimes \St_{\GG_{n-1}}(\Field;R) \cong \St_{\GL_1}(\Field;R) \otimes \St_{\GG_{n-1}}(\Field;R) \rightarrow \St_{\GG_n}(\Field;R),\]
where the final arrow is the Reeder product map.  Here we are using the convention regarding
the empty set discussed at the end of the introduction which implies that
$\St_{\GL_1}(\Field;R) = R$; this convention is compatible with Theorem \ref{theorem:reeder}.
\end{compactitem}
The main result of this section is then as follows.

\begin{lemma}
\label{lemma:reducetostab}
Let $\GG_n$ be either $\GL_{n}$, $\SL_{n}$, $\Sp_{2n}$, $\SO_{n,n}$, or $\SO_{n,n+1}$.
Let $\Field$ be a field and let $R$ be a commutative ring.  Assume that the stabilization
map \eqref{eqn:stabmap}
is a surjection for $n \geq N$.  Then
$\HH_i(\GG_n(\Field);\St_{\GG_n}(\Field;R)) = 0$ for $n \geq N+1$.
\end{lemma}
\begin{proof}
Consider $n \geq N+1$.  By assumption, the map
\begin{equation}
\label{eqn:doublestab}
\HH_i(\GG_{n-2}(\Field);\St_{\GG_{n-2}}(\Field;R)) \rightarrow \HH_i(\GG_n(\Field);\St_{\GG_n}(\Field;R))
\end{equation}
obtained by iterating the stabilization map twice is surjective.  It is thus enough to show that the image of
this map is $0$.  We can factor this map as
\begin{align*}
\HH_i(\GG_{n-2}(\Field);\St_{\GG_{n-2}}(\Field;R)) &\rightarrow \HH_i(\GG_{n-2}(\Field);\St_{\GL_2}(\Field;R) \otimes \St_{\GG_{n-2}}(\Field;R)) \\
&\rightarrow \HH_i(\GG_n(\Field);\St_{\GG_n}(\Field;R)).
\end{align*}
Regard $\SL_2(\Field)$ as a subgroup
of $\GG_n(\Field)$ via the composition
\[
\SL_2(\Field) = \SL_2(\Field) \times 1 \hookrightarrow \Levi_n^2(\Field) \hookrightarrow \GG_n(\Field).
\]
The subgroup $\SL_2(\Field)$ of $\GG_n(\Field)$ commutes with the image of $\GG_{n-2}(\Field)$ 
in $\GG_n(\Field)$ under the map used to define \eqref{eqn:doublestab}.  Inner automorphisms
act trivially on homology, even with twisted coefficients; see \cite[Proposition III.8.1]{BrownCohomology}.
It follows that to show that the image of \eqref{eqn:doublestab} is $0$, it is enough to prove that
\[
(\St_{\GL_2}(\Field;R) \otimes \St_{\GG_{n-2}}(\Field;R))_{\SL_2(\Field)} = 0.
\]
This is equivalent to
\[
(\St_{\GL_2}(\Field;R))_{\SL_2(\Field)} = 0,
\]
which is one case of Lemma~\ref{lemma:killcoinvariants}.
\end{proof}

\section{The stabilizer subgroups}
\label{section:stabilizers}

This section constructs an isomorphism (Lemma \ref{lemma:vectstab} below)
that will play a fundamental role in
our proof of Theorem \ref{theorem:vanishing}.

Let $\GG_n$ be either $\GL_{n}$, $\SL_{n}$, $\Sp_{2n}$, $\SO_{n,n}$, or $\SO_{n,n+1}$. 
Let $\Field$ be a field and $R$ be a commutative
ring.  Fix some $1 \leq \ell \leq n$.  There is a map
\begin{equation}
\label{eqn:vectstab}
\HH_i(1 \times \GG_{n-\ell}(\Field);\St_{\GL_{\ell}}(\Field;R) \otimes \St_{\GG_{n-\ell}}(\Field;R)) \rightarrow \HH_i(\Stab_n^{\ell}(\Field);\Res^{\GG_n(\Field)}_{\Stab_n^{\ell}(\Field)} \St_{\GG_n}(\Field;R)).
\end{equation}
induced by the following two maps:
\begin{compactitem}
\item The inclusion map $1 \times \GG_{n-\ell}(\Field) \rightarrow \Stab_n^{\ell}(\Field)$.
Here we are regarding $\Stab_n^{\ell}(\Field)$ as a subgroup of $\Para_n^{\ell}(\Field)$
that contains $1 \times \GG_{n-\ell}(\Field) \subset \Levi_n^{\ell}(\Field) \subset \Para_n^{\ell}(\Field)$.
\item The Reeder product map 
$\St_{\GL_{\ell}}(\Field;R) \otimes \St_{\GG_{n-\ell}}(\Field;R) \rightarrow \St_{\GG_n}(\Field;R)$.
\end{compactitem}
Our main result is as follows.

\begin{lemma}
\label{lemma:vectstab}
Let $\Field$ be a field, let $R$ be a commutative ring, let $1 \leq \ell \leq n$, and let $i \geq 0$.  Then the map
\eqref{eqn:vectstab} is an isomorphism.
\end{lemma}
\begin{proof}
Shapiro's Lemma 
\cite[Proposition III.6.2]{BrownCohomology} gives an isomorphism
\begin{align*}
&\HH_i(1 \times \GG_{n-\ell}(\Field);\St_{\GL_{\ell}}(\Field;R) \otimes \St_{\GG_{n-\ell}}(\Field;R)) \\
&\ \ \ \ \ \cong
\HH_i(\Stab_{n}^{\ell}(\Field);\Ind_{1 \times \GG_{n-\ell}(\Field)}^{\Stab_{n}^{\ell}(\Field)} \St_{\GL_{\ell}}(\Field;R) \otimes \St_{\GG_{n-\ell}}(\Field;R)).
\end{align*}
Below we will prove that there is an isomorphism
\begin{equation}
\label{eqn:toconstruct}
\Ind_{1 \times \GG_{n-\ell}(\Field)}^{\Stab_{n}^{\ell}(\Field)} 
\St_{\GL_{\ell}}(\Field;R) \otimes \St_{\GG_{n-\ell}}(\Field;R)
\cong
\Res^{\GG_n(\Field)}_{\Stab_n^{\ell}(\Field)} \St_{\GG_n}(\Field;R).
\end{equation}
of $\Stab_{n}^{\ell}(\Field)$-representations.  Combined with the above, this will yield an isomorphism between the left
and right hand sides of \eqref{eqn:vectstab} which is easily seen to be the map in \eqref{eqn:vectstab}.

It remains to construct the isomorphism \eqref{eqn:toconstruct}.  
Since 
$\Stab_{n}^{\ell}(\Field) \subset \Para_{n}^{\ell}(\Field)$,
we can restrict the isomorphism given by Theorem \ref{theorem:reeder} and Lemma \ref{lemma:identifystlevi} 
to obtain an isomorphism
\begin{equation}
\label{eqn:construction}
\Res_{\Stab_{n}^{\ell}(\Field)}^{\Para_{n}^{\ell}(\Field)} \Ind_{\Levi_{n}^{\ell}(\Field)}^{\Para_{n}^{\ell}(\Field)} \St_{\GL_{\ell}}(\Field;R) \otimes \St_{\GG_{n-\ell}}(\Field;R) \cong \Res^{\GG_n(\Field)}_{\Stab_{n}^{\ell}(\Field)} \St_{\GG_n}(\Field;R).
\end{equation}
The unipotent radical of $\Para_n^{\ell}(\Field)$ is contained in $\Stab_n^{\ell}(\Field)$.
This implies that there is a single $(\Stab_{n}^{\ell}(\Field),\Levi_{n}^{\ell}(\Field))$-double coset in $\Para_{n}^{\ell}(\Field)$.
Also, 
$\Stab_{n}^{\ell}(\Field) \cap \Levi_{n}^{\ell}(\Field) = \GG_{n-\ell}(\Field)$.
The double coset formula \cite[Proposition III.5.6b]{BrownCohomology} 
therefore implies that the left side of \eqref{eqn:construction} is canonically isomorphic to
\[\Ind_{1 \times \GG_{n-\ell}(\Field)}^{\Stab_{n}^{\ell}(\Field)} \Res^{\Levi_{n}^{\ell}(\Field)}_{1 \times \GG_{n-\ell}(\Field)} \St_{\GL_{\ell}}(\Field;R) \otimes \St_{\GG_{n-\ell}}(\Field;R)
= \Ind_{1 \times \GG_{n-\ell}(\Field)}^{\Stab_{n}^{\ell}(\Field)} 
\St_{\GL_{\ell}}(\Field;R) \otimes \St_{\GG_{n-\ell}}(\Field;R),\]
as desired.
\end{proof}

The following alternate version of Lemma \ref{lemma:vectstab} will be useful.

\begin{corollary}
\label{corollary:vectstab2}
Let $\Field$ be a field, let $R$ be a commutative ring, let $1 \leq \ell \leq n$, and let 
$i \geq 0$.  Then there exists an isomorphism 
\[\St_{\GL_{\ell}}(\Field;R) \otimes \HH_i(\GG_{n-\ell}(\Field);\St_{\GG_{n-\ell}}(\Field;R)) \cong \HH_i(\Stab_{n}^{\ell}(\Field);\Res^{\GG_n(\Field)}_{\Stab_{n}^{\ell}(\Field)} \St_{\GG_n}(\Field;R)).\]
\end{corollary}
\begin{proof}
Since $\St_{\GL_{\ell}}(\Field;R)$ is a free $R$-module, we have
\[\St_{\GL_{\ell}}(\Field;R) \otimes \HH_i(\GG_{n-\ell}(\Field);\St_{\GG_{n-\ell}}(\Field;R))
\cong
\HH_i(1 \times \GG_{n-\ell}(\Field);\St_{\GL_{\ell}}(\Field;R) \otimes \St_{\GG_{n-\ell}}(\Field;R)).\]
The corollary now follows from Lemma \ref{lemma:vectstab}.
\end{proof}

We will also need an explicit inverse 
\begin{equation}
\label{eqn:vectstabinv}
\HH_i(\Stab_n^{\ell}(\Field);\Res^{\GG_n(\Field)}_{\Stab_n^{\ell}(\Field)} \St_{\GG_n}(\Field;R))
\rightarrow 
\HH_i(1 \times \GG_{n-\ell}(\Field);\St_{\GL_{\ell}}(\Field;R) \otimes \St_{\GG_{n-\ell}}(\Field;R))
\end{equation}
to the isomorphism \eqref{eqn:vectstab}.  The map \eqref{eqn:vectstabinv} will
be induced by the following two maps:
\begin{compactitem}
\item The homomorphism $\Stab_n^{\ell}(\Field) \rightarrow 1 \times \GG_{n-\ell}(\Field)$
obtained by restricting the projection 
$\Para_n^{\ell}(\Field) \rightarrow \Levi_n^{\ell}(\Field)$ to $\Stab_n^{\ell}(\Field)$.
\item The map 
\[\St_{\GG_n}(\Field;R) \rightarrow \St_{\GL_{\ell}}(\Field;R) \otimes \St_{\GG_{n-\ell}}(\Field;R)\]
which equals the composition
\begin{align*}
\St_{\GG_n}(\Field;R) 
&= \bigoplus_{u \in \Uni_n^{\ell}(\Field)} u \cdot \left(\St_{\GL_{\ell}}\left(\Field;R\right) \otimes \St_{\GG_{n-\ell}}\left(\Field;R\right)\right) \\
&\rightarrow \St_{\GL_{\ell}}(\Field;R) \otimes \St_{\GG_{n-\ell}}(\Field;R),
\end{align*}
where the first equality comes from a combination of Theorem \ref{theorem:reeder} and
Lemma \ref{lemma:identifystlevi} (see Remark \ref{remark:reederdecompose})
and the last arrow takes 
$u \cdot x \in u \cdot (\St_{\GL_{\ell}}(\Field;R) \otimes \St_{\GG_{n-\ell}}(\Field;R))$
to $x \in \St_{\GL_{\ell}}(\Field;R) \otimes \St_{\GG_{n-\ell}}(\Field;R)$.
We will call this map the {\em Reeder projection map}.
\end{compactitem}
It is clear that these define a map of the form \eqref{eqn:vectstabinv}.  The following 
lemma says that this is an inverse to \eqref{eqn:vectstab}.

\begin{lemma}
\label{lemma:vectstabinv}
Let $\Field$ be a field, let $R$ be a commutative ring, let $1 \leq \ell \leq n$, and let $i \geq 0$.  Then the map
\eqref{eqn:vectstabinv} is an inverse to the map \eqref{eqn:vectstab}.
\end{lemma}
\begin{proof}
Immediate from the fact that the compositions
\[1 \times \GG_{n-\ell}(\Field) \rightarrow \Stab_n^{\ell}(\Field) \rightarrow 1 \times \GG_{n-\ell}(\Field)\]
and
\[\St_{\GL_{\ell}}(\Field;R) \otimes \St_{\GG_{n-\ell}}(\Field;R) \rightarrow \St_{\GG_n}(\Field;R) \rightarrow \St_{\GL_{\ell}}(\Field;R) \otimes \St_{\GG_{n-\ell}}(\Field;R)\]
of the maps used to define \eqref{eqn:vectstab} and \eqref{eqn:vectstabinv} equal the
identity.
\end{proof}

\section{Vanishing} \label{section:vanishing}

This section is devoted to the proof of Theorem \ref{theorem:vanishing}.  The actual
proof is in \S \ref{section:theproof}.  This is preceded by two sections of preliminaries.

\subsection{Equivariant homology} 
\label{section:equivariant}

In our proof of Theorem \ref{theorem:vanishing}, we will need some basic facts about
equivariant homology.  A basic reference is \cite[Chapter VII.7]{BrownCohomology}.

Let $G$ be a group, let $X$ be a semisimplicial set on which $G$ acts, let
$R$ be a ring, and let $M$ be an $R[G]$-module.  Let $EG$ be a contractible 
semisimplicial set on which $G$ acts freely and let $BG = EG/G$, so $BG$ is
a classifying space for $G$.  Denote by
$EG \times_G X$ the quotient of $EG \times X$ by the diagonal action of $G$. This is known as the {\em Borel construction}.  The homotopy type of $EG \times_G X$ does not depend on the choice of $EG$.  The projection $EG \times_G X \rightarrow EG/G = BG$
induces a homomorphism $\pi_1(EG \times_G X) \rightarrow \pi_1(BG) = G$.  Via
this homomorphism, we can regard $M$ as a local coefficient system on $EG \times_G X$.
The {\em $G$-equivariant homology groups} of $X$ with coefficients in $M$, denoted
$\HH_{\ast}^G(X;M)$, are the homology groups of $EG \times_G X$ with respect to
the local coefficient system $M$.

\begin{lemma}
\label{lemma:equivariantcalculates}
If $X$ is $\ell$-connected, then the above map $EG \times_G X \rightarrow EG/G = BG$
induces an isomorphism $\HH_i^G(X;M) \cong \HH_i(G;M)$ for $0 \leq i \leq \ell$ and a surjection $\HH_{\ell+1}^G(X;M) \to \HH_{\ell+1}(X;M)$.
\end{lemma}
\begin{proof}
The group $G$ acts freely on $EG \times X$ and $EG \times X$ is $\ell$-connected.  Viewing
$EG \times X$ as a CW-complex, we
can make $EG \times X$ contractible by adding cells of dimension at least $(\ell+2)$.
We conclude that there exists a classifying space for $G$ whose $(\ell+1)$-skeleton
equals the $(\ell+1)$-skeleton of $EG \times_G X$.  The lemma follows.
\end{proof}

Our main tool for understanding $\HH_{\ast}^G(X;M)$ is the following spectral
sequence, which is constructed in \cite[Equation VII.7.7]{BrownCohomology}.

\begin{lemma}
\label{lemma:spectralsequence}
For all $p \geq 0$, let $\Sigma_p$ be a set containing exactly one representative
for each orbit of the action of $G$ on the $p$-simplices of $X$.  For $\sigma \in \Sigma_p$,
let $G_{\sigma}$ be the stabilizer of $\sigma$.  
Then there is a first quadrant spectral sequence 
\[
E^1_{p,q} = \bigoplus_{\sigma \in \Sigma_p} \HH_q(G_{\sigma};\Res^G_{G_{\sigma}} M) \Longrightarrow \HH_{p+q}^G(X;M).
\]
\end{lemma}

\begin{remark}
In \cite[Equation VII.7.7]{BrownCohomology}, the action of $G_\sigma$ on $M$ is
twisted by an ``orientation character''; however, this is unnecessary
in our situation,
 since we
are working with semisimplicial sets rather than ordinary simplicial complexes
(the point being that in the geometric realization, the setwise stabilizer of a
simplex stabilizes the simplex pointwise).
\end{remark}

\subsection{Complexes of partial bases}
\label{section:complex}

Let $\Field$ be a field and let $\GG_n$ be either $\GL_{n}$, $\SL_{n}$, $\Sp_{2n}$, $\SO_{n,n}$, or $\SO_{n,n+1}$.  
To prove Theorem \ref{theorem:vanishing}, we will need to construct a highly connected space
$\Cpx_n(\Field)$ on which $\GG_n(\Field)$ acts.  The definition of this complex is
as follows.
\begin{compactitem}
\item If $\GG_n = \GL_n$ or $\GG_n = \SL_n$, then 
define $\Cpx_n(\Field)$ to be the {\em complex of partial bases} for $\Field^n$, 
i.e.\ the semisimplicial
complex whose $\ell$-simplices are ordered sequences $[\vec{v}_0,\ldots,\vec{v}_{\ell}]$ 
of linearly independent elements of $\Field^n$.
\item If $\GG_n = \Sp_{2n}$ or $\GG_n = \SO_{n,n}$ or $\GG_n = \SO_{n,n+1}$ and
$\Field^m$ is the vector space upon which $\GG_n(\Field)$ acts (so $m$ is either $2n$
or $2n+1$), then define $\Cpx_n(\Field)$ to be the {\em complex of partial isotropic
bases} for $\Field^m$, i.e.\ the semisimplicial
complex whose $\ell$-simplices are ordered sequences $[\vec{v}_0,\ldots,\vec{v}_{\ell}]$ of 
linearly independent elements of $\Field^m$ that span an isotropic subspace.
\end{compactitem}
The following theorem summarizes the properties of $\Cpx_n(\Field)$.

\begin{theorem}
\label{theorem:complexproperties}
Let $\Field$ be a field and let $\GG_n$ be either 
$\GL_{n}$, $\SL_{n}$, $\Sp_{2n}$, $\SO_{n,n}$, or $\SO_{n,n+1}$.  The following then
hold.
\begin{compactenum}[\indent \rm 1.]
\item The group $\GG_n(\Field)$ acts transitively on the $\ell$-cells of $\Cpx_n(\Field)$
for all $0 \leq \ell < n-1$.
\item The space $\Cpx_n(\Field)$ is $f(n)$-connected where $f(n)$ is given by:
  \begin{enumerate}[\rm (a)]
  \item $f(n) = n-2$ if $\GG_n$ is either $\GL_n$ or $\SL_n$,
  \item $f(n) = \frac{n-3}{2}$ if $\GG_n$ is either $\Sp_{2n}$, $\SO_{n,n}$, or $\SO_{n,n+1}$.
  \end{enumerate}
\end{compactenum}
\end{theorem}
\begin{proof}
The first assertion is well known (and also holds for $\ell = n-1$ except when $\GG_n = \SL_n$). As for the second, Maazen proved in his thesis \cite{MaazenThesis} that $\Cpx_n(\Field)$ is $(n-2)$-connected for $\GG_n = \GL_n$ and $\GG_n = \SL_n$.  See \cite{VDK} for a published proof of a more general result. Friedrich proved in \cite[Theorem 3.23]{FriedrichConn} 
that $\Cpx_n(\Field)$ is $\frac{n-3}{2}$-connected for $\GG_n = \Sp_{2n}$ and $\GG_n = \SO_{n,n}$ and $\GG_n = \SO_{n,n+1}$
(for $\Sp_{2n}$ and $\SO_{n,n}$, this was proven earlier in \cite[Theorem 7.3]{mvdk}). To apply the cited result of Friedrich to our situation, we need the fact that the unitary stable rank of a field is $1$ (see, e.g., \cite[Example 6.5]{mvdk}).
\end{proof}

\subsection{The proof of Theorem \ref{theorem:vanishing}} 
\label{section:theproof}

Let us first recall the statement of the theorem.
Let $\GG_n$ be either $\GL_{n}$, $\SL_{n}$, $\Sp_{2n}$, $\SO_{n,n}$, or $\SO_{n,n+1}$.
Also, let $\Field$ be a field and $R$ be a commutative ring.  Our goal is to prove that
$\HH_i(\GG_n(\Field); \St_{\GG_n}(\Field;R)) = 0$ for $n \geq 2i+2$ and that there exists a surjection
\begin{align} \label{eqn:toprovesurjective1}
\HH_i(\GG_{2i}(\Field);\St_{\GG_{2i}}(\Field;R)) \to \HH_i(\GG_{2i+1}(\Field); \St_{\GG_{2i+1}}(\Field; R)).
\end{align}
Of course, this surjection will be induced by the stabilization map defined in
\S \ref{section:stability}.

The proof is by induction on $i$. We begin with the base case $i=0$.  Lemma \ref{lemma:killcoinvariants} says
that $\HH_0(\GG_{n}(\Field); \St_{\GG_{n}}(\Field; R)) = 0$ for $n \geq 2$, so we only need to show that
the map \eqref{eqn:toprovesurjective1} is a surjection for $i=0$.
For the domain, $\GG_0(\Field)$ is the trivial group.  By our convention regarding the
empty set discussed at the end of the introduction, we thus have $\St_{\GG_0}(\Field; R) = R$,
and hence $\HH_i(\GG_0(\Field);\St_{\GG_0}(\Field;R)) = R$. 
To simplify the codomain, we have several cases.
\begin{compactitem}
\item $\GG_1 = \GL_1$ or $\GG_1 = \SO_{1,1}$.  In fact, these groups are
isomorphic and are commutative, so
$\St_{\GG_1}(\Field;R) = R$ in these cases and \eqref{eqn:toprovesurjective1} is an
isomorphism.
\item $\GG_1 = \SL_1$.  The group $\SL_1$ is the trivial group and thus 
$\St_{\GG_1}(\Field;R) = R$ and \eqref{eqn:toprovesurjective1} is an
isomorphism.
\item $\GG_1 = \Sp_2 \cong \SL_2$ or $\GG_1 = \SO_{2,1} \cong \PSL_2$.  
These groups have isomorphic Steinberg representations and the action of $\SL_2(\Field)$
on $\St_{\SL_2}(\Field;R)$ factors through $\PSL_2(\Field)$.
This case thus follows from Lemma \ref{lemma:killcoinvariants}, which says that $\HH_0(\SL_2(\Field);\St_{\SL_2}(\Field;R)) = 0$.
\end{compactitem}
This completes the base case.

Assume now that $i>0$ and that the desired result is true for all smaller values of $i$.
We will prove that the stabilization map
\begin{equation}
\label{eqn:toprovesurjective2}
\HH_i(\GG_{n-1}(\Field);\St_{\GG_{n-1}}(\Field;R)) \to \HH_i(\GG_{n}(\Field); \St_{\GG_{n}}(\Field; R))
\end{equation}
is surjective for $n \geq 2i+1$.  Lemma \ref{lemma:reducetostab} will then imply that $\HH_i(\GG_{n}(\Field); \St_{\GG_{n}}(\Field; R)) = 0$ for $n \geq 2i+2$, and the theorem will follow.

Fix some $n \geq 2i+1$ and let $\Field^m$ be the standard vector space representation of $\GG_n(\Field)$ (so $m$ is either $n$, $2n$, or $2n+1$).  
Let $\{\vec{a}_1,\ldots,\vec{a}_n\}$ be the vectors in $\Field^m$ such that
\[
\Stab_n^{\ell}(\Field) = \Set{$M \in \GG_n(\Field)$}{$M(\vec{a}_j) = \vec{a}_j$ for $1 \leq j \leq \ell$}
\]
for $1 \leq \ell \leq n$.
Combining the second conclusion of Theorem~\ref{theorem:complexproperties} with Lemma~\ref{lemma:equivariantcalculates}, we have a surjection
\begin{equation}
\label{eqn:identifyequivariant}
\HH_i^{\GG_n(\Field)}(\Cpx_n(\Field);\St_{\GG_n}(\Field; R)) \to \HH_i(\GG_n(\Field); \St_{\GG_n}(\Field; R)).
\end{equation}
We will analyze $\HH_i^{\GG_n(\Field)}(\Cpx_n(\Field);\St_{\GG_n}(\Field; R))$ using
the spectral sequence from Lemma \ref{lemma:spectralsequence}.  To calculate
its $E^1$-page, observe that the first conclusion of Theorem \ref{theorem:complexproperties}
says that $\GG_n(\Field)$ acts transitively on the $p$-simplices of $\Cpx_n(\Field)$ for $0 \leq p < n-1$.
The stabilizer of the $(p-1)$-simplex
$[\vec{a}_{1},\ldots,\vec{a}_p]$ is $\Stab_n^p(\Field)$, so the spectral sequence in Lemma \ref{lemma:spectralsequence} has
\begin{equation}
\label{eqn:e1description}
E^1_{p,q} = \HH_q(\Stab_{n}^{p+1}(\Field);\Res^{\GG_n(\Field)}_{\Stab_{n}^{p+1}(\Field)} \St_{\GG_n}(\Field;R)),
\end{equation}
for $0 \leq p < n-1$.

We will prove that all of the terms on the $p+q = i$ line of the $E^{\infty}$-page of our spectral sequence vanish except for possibly the term $E^{\infty}_{0,i}$.  To do this, consider $p,q \geq 0$ with $p+q=i$ and $p \geq 1$.  The case $p=1$ and $n=2i+1$ is exceptional and must be treated separately.  To avoid getting bogged down here,
we postpone this calculation until \S \ref{section:differential} below, where it appears as Lemma \ref{lemma:killexceptional}.\footnote{This exceptional case could be avoided at the cost of only proving that $\HH_i(\GG_n;\St_{\GG_n}(\Field;R)) = 0$ for $n \geq 3i+2$ instead of for $n \geq 2i+2$.}

We thus can assume that either $p \geq 2$ or that $n \geq 2i+2$.
Since $n \geq 2i+1$, we certainly have $p<n-1$, so $E^1_{p,q}$ is in the regime where the above description of the $E^1$-page holds.  Applying Corollary~\ref{corollary:vectstab2} to \eqref{eqn:e1description}, we see that 
\[
E^1_{p,q} = \St_{\GL_{p+1}}(\Field;R) \otimes \HH_q(\GG_{n-p-1}(\Field);\St_{\GG_{n-p-1}}(\Field;R)).
\]
To see that this vanishes, it is enough to show that
$\HH_q(\GG_{n-p-1}(\Field);\St_{\GG_{n-p-1}}(\Field;R)) = 0$.  This is a
consequence of our inductive hypothesis; to see that it applies, observe that if $p \geq 2$ then
\[
n-p-1 \geq (2i+1) - p - 1 = 2(p+q)-p = 2q+p \geq 2q+2,
\]
while if $p = 1$ and $n \geq 2i+2$ then
\[
n-p-1 \geq (2i+2) - 1 - 1 = 2(p+q) = 2q+2.
\]
This implies that $E^1_{p,q} = 0$, and thus that $E^{\infty}_{p,q} = 0$.

The $p+q=i$ line of the $E^\infty$-page of our spectral sequence thus only has a single potentially nonzero entry, namely $E^\infty_{0,i}$, and this is a quotient of 
\[
E^1_{0,i} = \HH_i(\Stab_{n}^{1}(\Field);\Res^{\GG_n(\Field)}_{\Stab_{n}^{1}(\Field)} \St_{\GG_n}(\Field; R)).
\]
This entry thus surjects onto $\HH_i^{\GG_n(\Field)}(\Cpx_n(\Field);\St_{\GG_n}(\Field; R))$.
Combining this with the surjection \eqref{eqn:identifyequivariant}, we obtain
a surjection
\begin{equation}
\label{eqn:surjection}
\HH_i(\Stab_{n}^{1}(\Field);\Res^{\GG_n(\Field)}_{\Stab_{n}^{1}(\Field)} \St_{\GG_n}(\Field; R)) \longrightarrow \HH_i(\GG_n(\Field);\St_{\GG_n}(\Field; R)).
\end{equation}
Examining the construction of our spectral sequence in \cite[Chapter VII.7]{BrownCohomology},
it is easy to see that this comes from the map induced by the inclusion
$\Stab_{n}^{1}(\Field) \hookrightarrow \GG_n(\Field)$.  Combining
\eqref{eqn:surjection} with the isomorphism 
\[
\HH_i(\GG_{n-1}(\Field); \St_{\GG_{n-1}}(\Field; R)) \xrightarrow{\cong} \HH_i(\Stab_{n}^{1}(\Field);\Res^{\GG_n(\Field)}_{\Stab_{n}^{1}(\Field)} \St_{\GG_n}(\Field; R))
\]
given by the $\ell=1$ case of Corollary~\ref{corollary:vectstab2}, 
we conclude that \eqref{eqn:toprovesurjective2} is a surjection, as desired.

\section{Killing the exceptional term in the spectral sequence}
\label{section:differential}

This section is devoted to proving the vanishing result postponed from the proof of Theorem \ref{theorem:vanishing}
in \S \ref{section:theproof}.  The notation in this section is thus identical to that in 
\S \ref{section:theproof}:
\begin{compactitem}
\item $\GG_n$ is either $\GL_{n}$, $\SL_{n}$, $\Sp_{2n}$, $\SO_{n,n}$, or $\SO_{n,n+1}$.
\item $\Field$ is a field and $R$ is a commutative ring.
\item $i>0$ and $n = 2i+1$ (the only case that remained in that section).
\item $\Field^m$ is the standard vector space representation of
$\GG_n(\Field)$ (so $m$ is either $n$, $2n$, or $2n+1$).
\item $\{\vec{a}_1,\ldots,\vec{a}_n\}$ is the set of vectors in $\Field^m$ such that
\[
\Stab_n^{\ell}(\Field) = \Set{$M \in \GG_n(\Field)$}{$M(\vec{a}_j) = \vec{a}_j$ for $1 \leq j \leq \ell$}
\]
for $1 \leq \ell \leq n$.
\item $E^{r}_{p,q}$ is the spectral sequence from Lemma \ref{lemma:spectralsequence}
converging to $\HH_i^{\GG_n(\Field)}(\Cpx_n(\Field);\St_{\GG_n}(\Field;R))$.
\end{compactitem}
What we must prove is as follows.

\begin{lemma}
\label{lemma:killexceptional}
Let the notation be as above, and assume that the stabilization map
\begin{equation}
\label{eqn:killassumption}
\HH_{i-1}(\GG_{n-3}(\Field);\St_{\GG_{n-3}}(\Field;R)) \to \HH_{i-1}(\GG_{n-2}(\Field); \St_{\GG_{n-2}}(\Field; R))
\end{equation}
is surjective.  Then the differential $E^1_{2,i-1} \rightarrow E^1_{1,i-1}$ is surjective, and thus
$E^{\infty}_{1,i-1} = 0$.
\end{lemma}

The proof of Lemma \ref{lemma:killexceptional} is divided into five sections:
\begin{compactitem}
\item In \S \ref{section:differentialidentify}, we give an explicit form for the differential
$E^1_{2,i-1} \rightarrow E^1_{1,i-1}$.
\item In \S \ref{section:differentialisomorphism}, we translate that explicit form into one involving
the stabilization map \eqref{eqn:killassumption}.
\item In \S \ref{section:differentialsummary}, we summarize what remains to be proved.
\item In \S \ref{section:apartments}, we give some needed background information about
apartments.
\item In \S \ref{section:zetasurjective}, we finish off the proof of Lemma \ref{lemma:killexceptional}.
\end{compactitem}

\subsection{Identifying the differential}
\label{section:differentialidentify}

The notation is as in the beginning of \S \ref{section:differential}.  In this section, we identify
the differential $E^1_{2,i-1} \rightarrow E^1_{1,i-1}$.
Since $n = 2i+1$ and $i > 0$, we have $1 < n-1$, so $E^1_{1,i-1}$ is as described in \eqref{eqn:e1description}, i.e.\
\[E^1_{1,i-1} \cong  \HH_{i-1}(\Stab_{n}^{2}(\Field);\Res^{\GG_{n}(\Field)}_{\Stab_{n}^{2}(\Field)} \St_{\GG_{n}}(\Field;R)).\]
If $i=1$, then we do not have $2 < n-1$, so in this case $E^1_{2,i-1}$ is not
as described in \eqref{eqn:e1description}.  The issue is that $\GG_n(\Field)$ might not act
transitively on the $2$-simplices of $\Cpx_n(\Field)$ (this is actually only a problem
for $\GG_n = \SL_n$).  However, for all values of $i$ 
it is still the case that $E^1_{2,i-1}$ contains
\[\HH_{i-1}(\Stab_{n}^{3}(\Field);\Res^{\GG_{n}(\Field)}_{\Stab_{n}^{3}(\Field)} \St_{\GG_{n}}(\Field;R))\]
as a summand.  The restriction of the differential $E^1_{2,i-1} \rightarrow E^1_{1,i-1}$
to this summand is a map 
\begin{equation}
\label{eqn:describepartial}
\begin{split}
\partial\colon & \HH_{i-1}(\Stab_{n}^{3}(\Field);\Res^{\GG_{n}(\Field)}_{\Stab_{n}^{3}(\Field)} \St_{\GG_{n}}(\Field;R)) \longrightarrow \HH_{i-1}(\Stab_{n}^{2}(\Field);\Res^{\GG_{n}(\Field)}_{\Stab_{n}^{2}(\Field)} \St_{\GG_{n}}(\Field;R)). 
\end{split}
\end{equation}
To prove Lemma~\ref{lemma:killexceptional}, it is enough to prove that 
$\partial$ is surjective.

We can describe $\partial$ using the recipe described in \cite[Chapter VII.8]{BrownCohomology}.  Recall that $\Stab_{n}^{3}(\Field)$ is the $\GG_{n}(\Field)$-stabilizer of the ordered sequence of vectors 
$\sigma = [\vec{a}_1,\vec{a}_2,\vec{a}_3]$.  For $1 \leq m \leq 3$, let $\sigma_m$ be the ordered
sequence obtained by deleting $\vec{a}_m$ from $\sigma$ and let
$(\GG_n(k))_{\sigma_m}$ denote the $\GG_n(k)$-stabilizer of $\sigma_m$.  We then have $\partial = \partial_1 - \partial_2 + \partial_3$, where
$\partial_m$ is the composition
\begin{align*}
&\HH_{i-1}(\Stab_{n}^{3}(\Field);\Res^{\GG_{n}(\Field)}_{\Stab_{n}^{3}(\Field)} \St_{\GG_{n}}(\Field;R))\\
&\ \ \ \ \ \stackrel{\partial_m'}{\longrightarrow} \HH_{i-1}((\GG_{n}(\Field))_{\sigma_m};\Res^{\GG_{n}(\Field)}_{(\GG_{n}(\Field))_{\sigma_m}} \St_{\GG_{n}}(\Field;R))\\
&\ \ \ \ \ \stackrel{\partial_m''}{\longrightarrow} \HH_{i-1}(\Stab_{n}^{2}(\Field);\Res^{\GG_{n}(\Field)}_{\Stab_{n}^{2}(\Field)} \St_{\GG_{n}}(\Field;R))
\end{align*}
of the following two maps.
\begin{compactitem}
\item $\partial_m'$ is the map induced by the inclusion 
$\Stab_{n}^{3}(\Field) \hookrightarrow (\GG_{n}(\Field))_{\sigma_m}$.
\item Define $\kappa_m \in \GG_n(\Field)$ as follows.  First, $\kappa_3 = \text{id}$.  For $m \in \{1,2\}$, we do the following.
\begin{compactitem}
\item If $\GG_{n} = \GL_{n}$ or $\GG_{n} = \SL_{n}$, then $\kappa_m \in \SL_{n}(\Field)$ is the map $\Field^{n} \rightarrow \Field^{n}$ that takes $\vec{a}_m$ to
$\vec{a}_3$, takes $\vec{a}_3$ to $-\vec{a}_m$, and fixes all the other basis vectors.
\item If $\GG_{n} = \Sp_{2n}$ or $\GG_{n} = \SO_{n,n}$ or $\GG_{n} = \SO_{n,n+1}$, then $\kappa_m \in \GG_{n}(\Field)$ is
the map $\Field^m \rightarrow \Field^m$ defined as follows.  Let $\vec{b}_1,\ldots,\vec{b}_{n}$ be the standard
basis vectors for $\Field^m$ that pair with the $\vec{a}_j$ (there is one additional standard basis vector if
$\GG_{n} = \SO_{n,n+1}$).  Then $\kappa_m$ takes $\vec{a}_m$ to $\vec{a}_3$, takes
$\vec{a}_3$ to $-\vec{a}_m$, takes $\vec{b}_m$ to $\vec{b}_3$, takes $\vec{b}_3$ to 
$-\vec{b}_m$, and fixes all the other basis vectors.  
\end{compactitem}
Then $\partial_m''$ is induced by the map $(\GG_{n}(\Field))_{\sigma_m} \rightarrow \Stab_{n}^2(\Field)$ that
takes $g \in (\GG_{n}(\Field))_{\sigma_m}$ to $\kappa_m g \kappa_m^{-1}$ and the map
$\St_{\GG_{n}}(\Field;R) \rightarrow \St_{\GG_{n}}(\Field;R)$ that takes
$x \in \St_{\GG_{n}}(\Field;R)$ to $\kappa_m(x) \in \St_{\GG_{n}}(\Field;R)$.
We remark that easier choices of $\kappa_m$ (without the signs) could be used
for $\GG_n \neq \SL_n$, but we chose the ones above to make our later formulas more uniform.
\end{compactitem}
This is summarized in the following lemma.

\begin{lemma}
\label{lemma:describedifferential}
Let the notation be as above.  Then the map $\partial$ in \eqref{eqn:describepartial} equals
$\partial_1 - \partial_2 + \partial_3$, where $\partial_m$ is induced by the
map $\Stab_{n}^3(\Field) \rightarrow \Stab_{n}^2(\Field)$ defined via the formula
\[g \mapsto \kappa_m g \kappa_m^{-1} \quad \quad (g \in \Stab_{n}^3(\Field))\]
and the map $\St_{\GG_{n}}(\Field;R) \rightarrow \St_{\GG_{n}}(\Field;R)$ defined via the formula
\[x \mapsto \kappa_m(x) \quad \quad (x \in \St_{\GG_{n}}(\Field;R)).\]
\end{lemma}

\subsection{Bringing in the stabilization map}
\label{section:differentialisomorphism}

The notation is as in the beginning of \S \ref{section:differential}.  Fix some $1 \leq m \leq 3$, and let
$\partial_m$ and $\kappa_m$ be as in Lemma \ref{lemma:describedifferential}.  Applying the isomorphism
in Corollary \ref{corollary:vectstab2} to the domain and codomain of $\partial_m$, we obtain a homomorphism
\begin{align*}
\hpartial_m\colon &\St_{\GL_3}(\Field;R) \otimes \HH_{i-1}(\GG_{n-3}(\Field);\St_{\GG_{n-3}}(\Field;R))\\
&\ \ \ \ \ \rightarrow \St_{\GL_2}(\Field;R) \otimes \HH_{i-1}(\GG_{n-2}(\Field);\St_{\GG_{n-2}}(\Field;R)).
\end{align*}
Our goal in this section is to prove that $\hpartial_m$ is the tensor product of the stabilization
map 
\begin{equation}
\label{eqn:stabilizationmappartial}
\HH_{i-1}(\GG_{n-3}(\Field);\St_{\GG_{n-3}}(\Field;R)) \rightarrow \HH_{i-1}(\GG_{n-2}(\Field);\St_{\GG_{n-2}}(\Field;R))
\end{equation}
with the map $\zeta_m\colon \St_{\GL_3}(\Field;R) \rightarrow \St_{\GL_2}(\Field;R)$
defined as follows.  Let $\{\vec{a}_1,\vec{a}_2,\vec{a}_3\}$ be the standard
basis for $\Field^3$.  Let $\hkappa_3 = \text{id} \in \SL_3(\Field)$, 
and for $m \in \{1,2\}$ let
$\hkappa_m \in \SL_3(\Field)$ be the element that 
takes $\vec{a}_m$ to $\vec{a}_3$, takes $\vec{a}_3$ to $-\vec{a}_m$, and fixes
all the other basis vectors.  Then $\zeta_m$ is
the composition
\[\St_{\GL_3}(\Field;R) \stackrel{\hkappa_m}{\longrightarrow} \St_{\GL_3}(\Field;R)
\longrightarrow \St_{\GL_2}(\Field;R) \otimes \St_{\GL_1}(\Field;R) \cong \St_{\GL_2}(\Field;R),\]
where the second arrow is the Reeder projection map (see \S \ref{section:stabilizers}) and
the final isomorphism comes from the fact that $\St_{\GL_1}(\Field;R) = R$.

The main result of this section is then as follows.

\begin{lemma}
\label{lemma:identifypartiali}
Let the notation be as above.  Then $\hpartial_m$ is the tensor product of $\zeta_m$ with the 
stabilization map \eqref{eqn:stabilizationmappartial}.
\end{lemma}
\begin{proof}
By construction, $\hpartial_m$ equals the composition
\begin{align*}
&\St_{\GL_3}(\Field;R) \otimes \HH_{i-1}(\GG_{n-3}(\Field);\St_{\GG_{n-3}}(\Field;R))\\
&\ \ \ \ \ \stackrel{\cong}{\longrightarrow} \HH_{i-1}(1 \times \GG_{n-3}(\Field);\St_{\GL_3}(\Field;R) \otimes \St_{\GG_{n-3}}(\Field;R))\\
&\ \ \ \ \ \stackrel{\cong}{\longrightarrow} \HH_{i-1}(\Stab_n^3(\Field);\St_{\GG_n}(\Field;R))\\
&\ \ \ \ \ \stackrel{\partial_m}{\longrightarrow} \HH_{i-1}(\Stab_n^2(\Field);\St_{\GG_n}(\Field;R))\\
&\ \ \ \ \ \stackrel{\cong}{\longrightarrow} \HH_{i-1}(1 \times \GG_{n-2}(\Field);\St_{\GL_2}(\Field;R) \otimes \St_{\GG_{n-2}}(\Field;R))\\
&\ \ \ \ \ \stackrel{\cong}{\longrightarrow} \St_{\GL_2}(\Field;R) \otimes \HH_{i-1}(\GG_{n-2}(\Field);\St_{\GG_{n-2}}(\Field;R))
\end{align*}
where the various maps are as follows:
\begin{compactitem}
\item The first and last arrows use the fact that $\St_{\GL_3}(\Field;R)$ and
$\St_{\GL_2}(\Field;R)$ are free $R$-modules (cf.\ the proof of Corollary
\ref{corollary:vectstab2}).
\item The second arrow is the map described in Lemma \ref{lemma:vectstab}, that
is, the map induced by the inclusion 
$1 \times \GG_{n-3}(\Field) \hookrightarrow \Stab_n^3(\Field)$ and the Reeder
product map 
$\St_{\GL_3}(\Field;R) \otimes \St_{\GG_{n-3}}(\Field;R) \rightarrow \St_{\GG_n}(\Field;R)$.
\item The third arrow is the map $\partial_m$ described in Lemma \ref{lemma:describedifferential}, that is, the map induced by the map $\Stab_n^3(\Field) \rightarrow \Stab_n^2(\Field)$
given by conjugation by $\kappa_m$ and the map $\St_{\GG_n}(\Field; R) \rightarrow \St_{\GG_n}(\Field; R)$
induced by $\kappa_m$.
\item The fourth arrow is the map described in Lemma \ref{lemma:vectstabinv}, that is,
the map induced by the projection 
$\Stab_n^2(\Field) \rightarrow 1 \times \GG_{n-2}(\Field)$ together with the Reeder
projection map 
$\St_{\GG_n}(\Field;R) \rightarrow \St_{\GL_2}(\Field;R) \otimes \St_{\GG_{n-2}}(\Field;R)$.
\end{compactitem}
We must show that this composition equals the indicated tensor product of maps.  This will take some work.

Define $\Psi$ to be the composition
\begin{align*}
\St_{\GL_3}(\Field;R) \otimes \St_{\GG_{n-3}}(\Field;R)
&\longrightarrow \St_{\GG_n}(\Field;R) \\
&\stackrel{\kappa_m}{\longrightarrow} \St_{\GG_n}(\Field;R) \\
&\longrightarrow \St_{\GL_2}(\Field;R) \otimes \St_{\GG_{n-2}}(\Field;R),
\end{align*}
where the first map is the Reeder product map and the last map is the Reeder projection
map.  Also, define $\Phi$ to be the composition
\begin{align*}
\St_{\GL_3}(\Field;R) \otimes \St_{\GG_{n-3}}(\Field;R)
&\longrightarrow \St_{\GL_2}(\Field;R) \otimes \St_{\GL_1}(\Field;R) \otimes \St_{\GG_{n-3}}(\Field;R) \\
&\longrightarrow \St_{\GL_2}(\Field;R) \otimes \St_{\GG_{n-2}}(\Field;R),
\end{align*}
where the maps are as follows:
\begin{compactitem}
\item The first map is the tensor product of the Reeder projection map
$\St_{\GL_3}(\Field;R) \rightarrow \St_{\GL_2}(\Field;R) \otimes \St_{\GL_1}(\Field;R)$
and the identity map $\St_{\GG_{n-3}}(\Field;R) \rightarrow \St_{\GG_{n-3}}(\Field;R)$.
\item The second map is the tensor product of the identity map
$\St_{\GL_2}(\Field;R) \rightarrow \St_{\GL_2}(\Field;R)$ 
and the Reeder product map 
$\St_{\GL_1}(\Field;R) \otimes \St_{\GG_{n-3}}(\Field;R) \rightarrow \St_{\GG_{n-2}}(\Field;R)$.
\end{compactitem}
By the above, it is enough to prove that $\Psi = \Phi \circ (\hkappa_m \otimes \text{id})$.

Define $\Psi'$ to be the composition
\begin{align*}
\St_{\GL_3}(\Field;R) \otimes \St_{\GG_{n-3}}(\Field;R)
&\longrightarrow \St_{\GG_n}(\Field;R) \\
&\longrightarrow \St_{\GL_2}(\Field;R) \otimes \St_{\GG_{n-2}}(\Field;R),
\end{align*}
where the first map is the Reeder product map and the second map is the Reeder
projection map.  From its definition, we see that 
$\Psi = \Psi' \circ (\hkappa_m \otimes \text{id})$.  We thus see that it is enough to
prove that $\Psi' = \Phi$.

Define $U = \UniGL_3^2(\Field)$ 
to be the unipotent radical of the parabolic subgroup $\ParaGL_3^2(\Field)$
of $\GL_3(\Field)$ (despite the bad notation, this is not the projective general linear
group).  
Using Theorem \ref{theorem:reeder} as in Remark \ref{remark:reederdecompose}, we see
that
\[\St_{\GL_3}(\Field;R) = \bigoplus_{u \in U} u \cdot \left(\St_{\GL_2}(\Field;R) \otimes \St_{\GL_1}(\Field;R)\right).\]
Consider $u \in U$ and $x \in \St_{\GL_2}(\Field;R)$ and $y \in \St_{\GL_1}(\Field;R)$
and $z \in \St_{\GG_{n-3}}(\Field;R)$.  Examining the definition of
$\Psi'$, we see that
\[\Psi'\left(\left(u \cdot \left(x \otimes y\right)\right) \otimes z\right) = x \otimes \left(y \otimes z\right),\]
where $y \otimes z \in \St_{\GL_1}(\Field;R) \otimes \St_{\GG_{n-3}}(\Field;R)$ is
identified with an element of $\St_{\GG_{n-2}}(\Field;R)$ using the Reeder product map.  But
this equals $\Phi\left(\left(u \cdot \left(x \otimes y\right)\right) \otimes z\right)$,
as desired.
\end{proof}

\subsection{Summary of where we are}
\label{section:differentialsummary}

The notation is as in the beginning of \S \ref{section:differential}.  Recall that
Lemma \ref{lemma:killexceptional} asserts that 
the differential $E^1_{2,i-1} \rightarrow E^1_{1,i-1}$
is surjective.  Let
$\partial$ be as in \S \ref{section:differentialidentify}.  Also, let $\zeta_m$
and $\hkappa_m$ be as in \S \ref{section:differentialisomorphism}.  Define
\[\zeta\colon \St_{\GL_3}(\Field;R) \rightarrow \St_{\GL_2}(\Field;R)\]
via the formula $\zeta = \zeta_1 - \zeta_2 + \zeta_3$.  Combining
Lemmas \ref{lemma:describedifferential} and \ref{lemma:identifypartiali}, we see
that to prove Lemma \ref{lemma:killexceptional}, it is enough to show that the map
\[\St_{\GL_3}(\Field;R) \otimes \HH_{i-1}(\GG_{n-3}(\Field);\St_{\GG_{n-3}}(\Field;R))
\rightarrow \St_{\GL_2}(\Field;R) \otimes \HH_{i-1}(\GG_{n-2}(\Field);\St_{\GG_{n-2}}(\Field;R))\]
obtained as the tensor product of $\zeta$ and the stabilization
map
\[\HH_{i-1}(\GG_{n-3}(\Field);\St_{\GG_{n-3}}(\Field;R)) \rightarrow \HH_{i-1}(\GG_{n-2}(\Field);\St_{\GG_{n-2}}(\Field;R))\]
is surjective.  One of the assumptions in Lemma \ref{lemma:killexceptional} is that this
stabilization map is surjective.  To prove that lemma, it is thus enough to prove the
following.

\begin{lemma}
\label{lemma:zetasurjective}
Let the notation be as above.  Then $\zeta$ is surjective.
\end{lemma}

\subsection{Apartments}
\label{section:apartments}

Before we prove Lemma \ref{lemma:zetasurjective}, we need to discuss some background
material on the Steinberg representation.  Unlike the previous sections, in this
section $n \geq 1$ is arbitrary.  Recall that 
$\St_{\GL_n}(\Field;R) = \RH_{n-2}(\Tits_{\GL_n}(\Field);R)$, where
$\Tits_{\GL_n}(\Field)$ is the Tits building associated to $\GL_n(\Field)$.  
This building can be described as the simplicial complex whose $r$-simplices
are flags
\[0 \subsetneq V_0 \subsetneq \cdots \subsetneq V_r \subsetneq \Field^n\]
of nonzero proper subspaces of $\Field^n$.

The Solomon--Tits theorem \cite[Theorem IV.5.2]{BrownBuildings} says that the $R$-module
$\St_{\GL_n}(\Field;R)$ is spanned by {\em apartment classes}, which are defined as follows.
Consider an $n \times n$ matrix $B$ with entries in $\Field^n$ none of whose columns are
identically $0$.  Let 
$(\vec{v}_1,\ldots,\vec{v}_n)$ be the columns of $B$.
Let $S_n$ be the simplicial complex whose $r$-simplices are chains
\[
0 \subsetneq I_0 \subsetneq \cdots \subsetneq I_r \subsetneq \{1,\ldots,n\}.
\]
The complex $S_n$ is isomorphic to the barycentric subdivision of the boundary of an $(n-1)$-simplex; in particular,
$S_n$ is homeomorphic to an $(n-2)$-sphere.  There is a simplicial map $f \colon S_n \rightarrow \Tits_{\GL_n}(\Field)$ defined
via the formula
\[
f(I) = \langle \text{$\vec{v}_i \mid i \in I$} \rangle \quad \quad (\emptyset \subsetneq I \subsetneq \{1,\ldots,n\}).
\]
The {\em apartment class} corresponding to $B$, denoted
$\Apartment{B}$, is the image of the fundamental class
$[S_n] \in \RH_{n-2}(S_n; R) = R$ under the map
$f_{\ast} \colon \RH_{n-2}(S_n; R) \rightarrow \RH_{n-2}(\Tits_{\GL_n}(\Field);R) = \St_{\GL_n}(\Field;R)$.

\begin{remark}
We have $\Apartment{B} = 0$ if the $\vec{v}_i$ 
do not form a basis for $\Field^n$, i.e.\ if $B$ is not invertible.
\end{remark}

Permuting the columns of $B$ changes $\Apartment{B}$ by the sign of the permutation, and
multiplying a column of $B$ by a nonzero scalar does not change $\Apartment{B}$.
The apartment classes also satisfy the following more interesting relation.

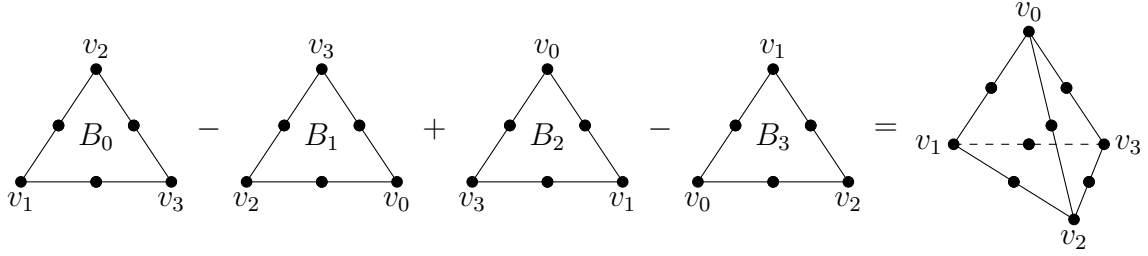
\begin{figure}
\begin{center}
\newcommand{\hide}[1]{#1}
\hide{
\begin{tikzpicture}
\newcommand*{\shift}{3}
\foreach \i in {0,1,2,3} {
%draw triangle
\coordinate (A) at (0+0.5*\shift*\i,0);
\coordinate (B) at (1+0.5*\shift*\i,0);
\coordinate (C) at (0.5+0.5*\shift*\i,0.75);
\draw ($2*(A)$) -- ($2*(B)$) -- ($2*(C)$) -- cycle;
\foreach \j in {(A), (B), (C)}
{ \foreach \k in {(A), (B), (C)}
{ \filldraw \j+\k circle (2pt);   }; };
%label vertices and triangles
\pgfmathsetmacro\resulta{int(mod(\i+1,4))};
\pgfmathsetmacro\resultb{int(mod(\i+2,4))};
\pgfmathsetmacro\resultc{int(mod(\i+3,4))};
\node[below] at ($2*(A)$) {$v_{\pgfmathprintnumber{\resulta}}$};
\node[below] at ($2*(B)$) {$v_{\pgfmathprintnumber{\resultc}}$};
\node[above] at ($2*(C)$) {$v_{\pgfmathprintnumber{\resultb}}$};
\node at (1+\shift*\i,0.6) {$B_\i$};
}
%arithmetic symbols
\node at (2.5,.7) {$-$};
\node at (5.5,.7) {$+$};
\node at (8.5,.7) {$-$};
\node at (11.5,.7) {$=$};
%final simplex
\renewcommand{\shift}{5.7}
\coordinate (D) at (\shift+1,1);
\coordinate (E) at (\shift+0.5,0.25);
\coordinate (F) at (\shift+1.3,-0.25);
\coordinate (G) at (\shift+1.5,0.25);
\draw ($2*(D)$) -- ($2*(E)$) -- ($2*(F)$) -- ($2*(G)$) -- cycle;
\draw ($2*(D)$) -- ($2*(F)$);
\foreach \j in {(D), (E), (F), (G)}
{ \foreach \k in {(D), (E), (F), (G)}
{ \filldraw \j+\k circle (2pt);   }; };
\draw[dashed] ($2*(E)$) -- ($2*(G)$);
\node[above] at ($2*(D)$) {$v_0$};
\node[left] at ($2*(E)$) {$v_1$};
\node[below] at ($2*(F)$) {$v_2$};
\node[right] at ($2*(G)$) {$v_3$};
\end{tikzpicture}}
\end{center}
\vspace{-.3in}
\caption{As we illustrate here in the case $n=3$, the apartment
classes corresponding to the $B_i$ can be placed on the boundary of an $n$-dimensional simplex such that their simplices
cancel in pairs.  In the picture, the vertices labeled with the vectors $\vec{v}_i$ are taken to the lines
spanned by the $\vec{v}_i$ while the unlabeled vertices are taken to the $2$-dimensional subspaces spanned
by the vectors on their two neighbors.}
\label{figure:apartmentscancel}
\end{figure}

\begin{lemma}
\label{lemma:relation}
Let $\Field$ be a field, let $R$ be a commutative ring, and let $n \geq 2$.  Let $B$
be an $n \times (n+1)$-matrix with entries in $\Field$.  Assume that none of the columns
of $B$ are identically $0$.  Ordering the columns of $B$ from $0$ to $n$, for 
$0 \leq m \leq n$ let $B_m$ be the result of deleting
the $m^{\text{th}}$ column from $B$.  Then
$\Apartment{B_0} - \Apartment{B_1} + \Apartment{B_2} - \cdots + (-1)^n \Apartment{B_{n}} = 0$.
\end{lemma}
\begin{proof}
The simplices forming the apartment classes $\Apartment{B_i}$ cancel in pairs; see Figure \ref{figure:apartmentscancel}.
\end{proof}

The Solomon--Tits theorem \cite[Theorem IV.5.2]{BrownBuildings} gives the following basis for $\St_{\GL_n}(\Field;R)$.

\begin{theorem}[Solomon--Tits]
\label{theorem:uppertriangular}
Let $\Field$ be a field, let $R$ be a commutative ring, and let $n \geq 1$.  Then $\St_{\GL_n}(\Field;R)$ is a free
$R$-module on the basis consisting of all $\Apartment{B}$ such that $B$ is an
upper unitriangular matrix in $\GL_n(\Field)$.
\end{theorem}

\subsection{The proof of Lemma \ref{lemma:zetasurjective}}
\label{section:zetasurjective}

We finally prove Lemma \ref{lemma:zetasurjective}, which as discussed in
\S \ref{section:differentialsummary} suffices to prove Lemma \ref{lemma:killexceptional}.
First, we recall its statement.
For $1 \leq m \leq 3$, let $\zeta_m$
and $\hkappa_m$ be as in \S \ref{section:differentialisomorphism}.  Define
\[\zeta\colon \St_{\GL_3}(\Field;R) \rightarrow \St_{\GL_2}(\Field;R)\]
via the formula $\zeta = \zeta_1 - \zeta_2 + \zeta_3$.
Our goal is to prove that $\zeta$ is surjective.

Before we do that, we introduce some formulas.  Let 
$\pi\colon \St_{\GL_3}(\Field;R) \rightarrow \St_{\GL_2}(\Field;R)$ be the composition
\[\St_{\GL_3}(\Field;R) \longrightarrow \St_{\GL_2}(\Field;R) \otimes \St_{\GL_1}(\Field;R) 
\stackrel{\cong}{\longrightarrow} \St_{\GL_2}(\Field;R),\]
where the first arrow is the Reeder projection map and the second arrow comes from the fact
that $\St_{\GL_1}(\Field;R) = R$.  From its definition, we see that
\[\pi(\Apartment{1 & x & y \\ 0 & 1 & z \\ 0 & 0 & 1}) = \Apartment{1 & x \\ 0 & 1}\]
for all $x,y,z \in \Field$.  What is more, for all $3 \times 3$ matrices $B$ none
of whose columns are identically $0$ we have
\[\zeta(\Apartment{B}) = \pi(\hkappa_1(B)) - \pi(\hkappa_2(B)) + \pi(\hkappa_3(B)).\]
Here the $\hkappa_m$ act on $B$ via matrix multiplication.

We now turn to proving that $\zeta$ is surjective.
Consider $a \in \Field$, and set
\[A_a = \Apartment{1 & a \\ 0 & 1}.\]
By Theorem \ref{theorem:uppertriangular}, it is enough to prove that
$A_a \in \Image(\zeta)$. We have
\begin{equation} \label{eqn:apartmentcalc}
\begin{split}
\zeta ( \Apartment{1 & a & 0 \\ 0 & 1 & 0 \\ 0 & 0 & 1})
&= \pi( \Apartment{0 & 0 & -1 \\ 0 & 1 & 0 \\ 1 & a & 0})
-  \pi( \Apartment{1 & a & 0 \\ 0 & 0 & -1 \\ 0 & 1 & 0})
+  \pi( \Apartment{1 & a & 0 \\ 0 & 1 & 0 \\ 0 & 0 & 1}) \\
&= \pi(-\Apartment{-1 & 0 & 0 \\ 0 & 1 & 0 \\ 0 & a & 1})
-  \pi(-\Apartment{1 & 0 & a \\ 0 & -1 & 0 \\ 0 & 0 & 1})
+  \pi( \Apartment{1 & a & 0 \\ 0 & 1 & 0 \\ 0 & 0 & 1}) \\
&= \pi(-\Apartment{1 & 0 & 0 \\ 0 & 1 & 0 \\ 0 & a & 1})
-  \pi(-\Apartment{1 & 0 & a \\ 0 & 1 & 0 \\ 0 & 0 & 1})
+  \pi( \Apartment{1 & a & 0 \\ 0 & 1 & 0 \\ 0 & 0 & 1}) \\
&= -\pi(\Apartment{1 & 0 & 0 \\ 0 & 1 & 0 \\ 0 & a & 1})
+ \Apartment{1 & 0 \\ 0 & 1}
+ \Apartment{1 & a \\ 0 & 1}, 
\end{split}
\end{equation}
where the second equality uses the fact that permuting the columns of a matrix
changes the associated apartment by the sign of the permutation and the third
equality uses the fact that multiplying a column by a nonzero scalar 
does not change the associated apartment.  

If $a=0$, then the right hand side of \eqref{eqn:apartmentcalc} simplifies to
\[-\Apartment{1 & 0 \\ 0 & 1}
+ \Apartment{1 & 0 \\ 0 & 1}
+ \Apartment{1 & 0 \\ 0 & 1} = \Apartment{1 & 0 \\ 0 & 1} = A_0,\]
so $A_0 \in \Image(\zeta)$.
Assume now that $a \neq 0$.  Plugging the matrix
\[\begin{pmatrix} 1 & 0 & 0 & 0 \\ 0 & 1 & 1 & 0 \\ 0 & 0 & a & 1 \end{pmatrix}\]
into Lemma \ref{lemma:relation}, we get the relation
\begin{align*}
0 &= \Apartment{0 & 0 & 0 \\ 1 & 1 & 0 \\ 0 & a & 1}
-    \Apartment{1 & 0 & 0 \\ 0 & 1 & 0 \\ 0 & a & 1}
+    \Apartment{1 & 0 & 0 \\ 0 & 1 & 0 \\ 0 & 0 & 1}
-    \Apartment{1 & 0 & 0 \\ 0 & 1 & 1 \\ 0 & 0 & a} \\
&= 0
- \Apartment{1 & 0 & 0 \\ 0 & 1 & 0 \\ 0 & a & 1}
+ \Apartment{1 & 0 & 0 \\ 0 & 1 & 0 \\ 0 & 0 & 1}
- \Apartment{1 & 0 & 0 \\ 0 & 1 & a^{-1} \\ 0 & 0 & 1},
\end{align*}
where the equality uses the fact that the columns of the first matrix are not linearly
independent and the fact that multiplying a column of a matrix by a nonzero scalar
does not change the associated apartment.  
Plugging this relation into \eqref{eqn:apartmentcalc}, we see that
the right hand side of \eqref{eqn:apartmentcalc} equals
\begin{align*}
&- \left(\pi(\Apartment{1 & 0 & 0 \\ 0 & 1 & 0 \\ 0 & 0 & 1}) - \pi(\Apartment{1 & 0 & 0 \\ 0 & 1 & a^{-1} \\ 0 & 0 & 1})\right) + \Apartment{1 & 0 \\ 0 & 1} + \Apartment{1 & a \\ 0 & 1} \\
=& -\Apartment{1 & 0 \\ 0 & 1} + \Apartment{1 & 0 \\ 0 & 1} + \Apartment{1 & 0 \\ 0 & 1} + \Apartment{1 & a \\ 0 & 1} \\
=& A_0 + A_a.
\end{align*}
Since we have already seen that $A_0 \in \Image(\zeta)$, we deduce that $A_a \in \Image(\zeta)$, as desired.

\begin{footnotesize}
\noindent
\begin{tabular*}{\linewidth}[t]{@{}p{\widthof{Department of Mathematics}+0.5in}@{}p{\widthof{Department of Mathematics}+0.5in}@{}p{\linewidth - \widthof{Department of Mathematics} - \widthof{Department of Mathematics} - 1in}@{}}
{\raggedright
Avner Ash\\
Department of Mathematics\\
Boston College\\
Chestnut Hill, MA 02467-3806\\
{\tt ashav@bc.edu}}
&
{\raggedright
Andrew Putman\\
Department of Mathematics\\
University of Notre Dame \\
279 Hurley Hall\\
Notre Dame, IN 46556\\
{\tt andyp@nd.edu}}
&
{\raggedright
Steven V Sam\\
Department of Mathematics\\
University of Wisconsin\\
480 Lincoln Dr.\\
Madison, WI 53706-1325\\
{\tt svs@math.wisc.edu}}
\end{tabular*}\hfill
\end{footnotesize}

\end{document}